\documentclass{article}

\usepackage{mathtext}
\usepackage[T2A]{fontenc}
\usepackage[cp1251]{inputenc}
\usepackage[english]{babel}

\usepackage{graphicx}
\usepackage{amsfonts,amssymb,amscd,amsmath,amsthm}

\usepackage{mathrsfs}
\usepackage{epsf}
\usepackage{wrapfig}
\usepackage{multirow}
\usepackage{subfigure}

\usepackage{latexsym}
\usepackage{indentfirst}
\usepackage{color}
\usepackage[colorlinks=true]{hyperref}
\usepackage[all]{xy}

\title{Vector bundles of finite rank on complete intersections
of finite codimension in ind-Grassmannians}

\author{Svetlana Ermakova}

\newtheorem{theorem}{Theorem}
\newtheorem{lemma}{Lemma}
\newtheorem{proposition}{Рroposition}
\newtheorem{zam}{Remark}
\newtheorem{definition}{Definition}
\newtheorem{corollary}{Corollary}

\begin{document}
\maketitle

\begin{abstract}
In this article we establish an analogue of 
the Barth–Van de Ven– Tyurin–Sato theorem.
We prove that a finite rank vector bundle on a complete intersection   of finite codimension in a linear ind-Grassmannian   is isomorphic to a direct
sum of line bundles.

\end{abstract}

\section{Introduction}

The Barth–Van de Ven–Tyurin–Sato theorem claims that any finite rank vector bundle on
the infinite complex projective space $\mathbf{ P^{\infty}}$ is isomorphic to a direct sum of line bundles.
For rank two bundles this was established by Barth and Van de Ven in
\cite{BV}, and for finite rank bundle
it was proved by Tyurin in \cite{T} and Sato in \cite{S1}.
In particular, the Barth–Van de Ven–Tyurin–Sato theorem
holds for linear ind-Grassmannians and their linear sections 
 $\cite{DP, PT1, PT2, PT3}$.

In this work we will extend these results on the case of a complete intersections in  linear ind-Grassmannians. The ground field in this work is $\mathbb{C}$.

First, we recall the definition 
of linear ind-varieties in general
and linear ind-Grassmanninans studied by 
I. Penkov and A. Tikhomirov 
in \cite{PT2,PT3,PT4} in particular.

\begin{definition}An  \textbf{ind-variety}  $\mathbf{X}=\underrightarrow{\lim}X_m$ 
is the direct limit of a chain of embeddings:
 \begin{equation}\label{X=} {\bf X}:=\{ X_1 \stackrel{\phi_1}{\hookrightarrow} X_2\stackrel{\phi_2}{\hookrightarrow} \ldots  \stackrel{\phi_{m-1}}{\hookrightarrow} X_m \stackrel{\phi_m}{\hookrightarrow} \ldots\},
\end{equation}
where $X_m$ is a smooth algebraic variety for any
 $m\geq1$.
\end{definition}

\begin{definition}
A \textbf{vector bundle} $\mathbf{E}$ of rank $r >0$ on $\mathbf{X}$ is the inverse limit
\begin{equation}\label{E=} \mathbf{E}=\underleftarrow{\lim}E_m
\end{equation}
of an inverse system
of vector bundles  $\{E_m\}_{m\geq 1}$ of rank $r$ on $\mathbf{X}$ $($i.e., a system of vector bundles
$E_m$ with fixed isomorphisms  $E_m\cong \phi_m^*E_{m+1}$$)$.
\end{definition}

In particular, the structure sheaf $\mathcal{O}_{\mathbf{X}}=\underleftarrow{\lim}\mathcal{O}_{X_m}$
of an ind-variety $\mathbf{X}$ is well defined.
By the Picard group $Pic\mathbf{X}$ we understand
the group of isomorphism classes of line bundles on $\mathbf{X}$.

\begin{definition}\label{linejnoe ind-mnogoobrazie} An ind-variety  $\mathbf{X}$  is called \textbf{linear},
if a line bundle $\mathcal{O}_{\mathbf{X}}(1)=\underleftarrow{\lim}\mathcal{O}_{X_m}(1)$ is defined on it, where for each $m\geq1$  the line bundle $\mathcal{O}_{X_m}(1)$ is ample on $X_m$.
\end{definition}

\subsubsection{Linear ind-Grassmannians and complete intersections in them}

For integers $m\geq1$, $n_m$ and $k_m$ satisfying  $1 \leq k_m \leq n_m$ consider the vector space  $V^{n_m}$ of dimension $n_m$ and the Grassmannian $G(k_m,n_m)$ of $k_m$-dimensional vector
subspaces in $V^{n_m}$. Consider as well the Pl{\"u}cker embedding of 
 $G(k_m,n_m)$:
$G(k_m,n_m){\hookrightarrow} \mathbb{P}^{N_m-1}=\mathbb{P}(\Lambda^{N_m})$, where 
$N_m=\binom{n_m}{k_m}$.

We define the ind-Grassmannian
 ${\bf{G:}}={\bf{G(\infty)}}$
 as the direct limit
 $\underrightarrow{\lim} G(k_m,n_m)$ of a chain of embeddings:
\begin{equation}\label{G=}  G(k_1,n_1) \stackrel{\phi_1}{\hookrightarrow} G(k_2,n_2)\stackrel{\phi_2}{\hookrightarrow} \ldots  \stackrel{\phi_{m-1}}{\hookrightarrow} G(k_m,n_m) \stackrel{\phi_m}{\hookrightarrow} \ldots,
\end{equation}
with conditions
$$\lim_{m\to \infty} k_m=\lim_{m\to \infty}(n_m-k_m)=\infty.$$

Further on we assume that the ind-Grassmannian
 $\bf G$ is \emph{linear},
i.e., it has a line bundle
$\mathcal{O}_{\mathbf{G}}(1)=\underleftarrow{\lim}\mathcal{O}_{G(k_m,n_m)}(1)$, 
where the class of the line bundle
$\mathcal{O}_{G(k_m,n_m)}(1)$ generates  $Pic(G(k_m,n_m))$ for all $m\geq 1$.

Since $\bf G$ is linear the embeddings
 $$\{\ldots  \stackrel{\phi_{m-1}}{\hookrightarrow} G(k_m,n_m) \stackrel{\phi_m}{\hookrightarrow} G(k_{m+1},n_{m+1}) \stackrel{\phi_{m+1}}{\hookrightarrow} \ldots\}$$
 can be extended to linear embeddings of Pl{\"u}cker spaces 
 $$\{ \ldots  \stackrel{\phi_{m-1}}{\hookrightarrow} \mathbb{P}^{N_m-1} \stackrel{\phi_m}{\hookrightarrow} \mathbb{P}^{N_{m+1}-1} \stackrel{\phi_{m+1}}{\hookrightarrow} \ldots\}.$$
 
Next, for $l \geq 1$  and $d_1,d_2,...,d_l \geq 1$
consider  the linear ind-subvariety   ${\bf X}$ of the linear ind-Grassmannian ${\bf G}$ that is the direct limit $\mathbf{X}=\underrightarrow{\lim}X_m$
of the following chain of embeddings

\begin{equation}\label{X=}
{\bf X}:=\{ X_1 \stackrel{\phi_1}{\hookrightarrow} X_2\stackrel{\phi_2}{\hookrightarrow} \ldots  \stackrel{\phi_{m-1}}{\hookrightarrow} X_m \stackrel{\phi_m}{\hookrightarrow} \ldots\}.
\end{equation}
Here $X_m$  is the intersection of the Grassmannian  $G(k_m,n_m)\subset \mathbb{P}^{N_m-1}$ with $l$  hypersurfaces $Y_{1,m}, Y_{2,m},\ldots,Y_{l,m}\subset \mathbb{P}^{N_m-1}$
of fixed degrees $\deg Y_{i,m} = d_i$, $i=1,...,l$, $m\geq 1$:

\begin{equation} X_m=G(k_m,n_m)\cap \bigcap_{i=1}^{l} Y_{i,m}, \; \; \; \mathrm{codim}_{G(k_m,n_m)}X_m=l.
 \end{equation}

\begin{definition} The constructed ind-variety  $\mathbf X$ is called a \textbf{complete intersection} of codimension  $l$ in the linear ind-Grassmannian $\mathbf{G}$.
\end{definition}

In other words, the ind-variety  $\mathbf{X}$ 
is an intersection of the ind-Grassmannian
$\mathbf{G}$ with ind-hypersufaces  $\mathbf{Y_1}, \mathbf{Y_2},\ldots,\mathbf{Y_l}$ in the linear ind-space $\mathbf{P}^{\infty}=\underrightarrow{\lim} \mathbb{P}^{N_m-1}$:
\begin{equation} \mathbf{X}=\mathbf{G}\cap \bigcap_{i=1}^{l} \mathbf{Y_i}, \; \mathbf{Y_i}={\underrightarrow{\lim}Y_{i,m}},\; i=1,...,l, \; {m\geq1}.
\end{equation}

For  $i=1,...,l$ number $\deg \mathbf{Y_i}:=\deg Y_{i,m}=d_i$, ${m\geq1}$ is called the \textit{ degree} of the ind-hyperusrace $\mathbf{Y_i}$ в $\mathbf{P}^{\infty}$.

The main result of this work is the following theorem.

\begin{theorem} \label{MainTheorem}
Any vector bundle of finite 
rank on a complete intersection  ${\bf X \subset \bf G}$ of finite codimension is isomorphic to a direct
sum of line bundles.
\end{theorem}

\textbf{Acknowledgement.} I would like to thank my 
PhD advisor A.S. Tikhomirov and 
A. Kuznetsov and D. Panov as well.

\section{Preliminary notions and the idea of proof
}

To explain the idea of proof of Theorem \ref{MainTheorem} we will need to give some definitions
and recall the main results of articles
\cite{ermakova} and \cite{ermakova2}.
\begin{definition}  \label{proektivnoe podmnogooobrazie} Let $X$ be a projective 
variety with an ample line bundle $\mathcal{O}_X(1)$.
A \textbf{projective space} in  $X$ 
a subvariety $M\simeq\mathbb{P}^r$ in $X$ such that $\mathcal{O}_ {X}(1)|_M\cong\mathcal{O}_{\mathbb{P}^r}(1)$.
In the case $\dim(M)=1$ we call $M$ 
a  \textbf{projective line} or just a \textbf{line} in $X$.
\end{definition}

Using Definition  \ref{proektivnoe podmnogooobrazie}
we can give the definition of projective subspace in 
an ind-variety $\mathbf{X}$.

\begin{definition} \label{proektivnoe podmnogooobrazie v X}
A \textbf{projective space} in $\mathbf{X}$ 
is a variety $M\simeq\mathbb{P}^r$ in $\mathbf{X}$, such that $\mathcal{O}_ {\mathbf{X}}(1)|_M\cong\mathcal{O}_{\mathbb{P}^r}(1)$.
\end{definition}

\begin{definition} A
\textbf{path} $p_n(x,y)$ of length $n$ 
on an ind-variety
$\mathbf{X}$ connecting points  $x$ and $y$, is a
collection of points  $x=x_0,x_1,...,x_n=y$ in
$\mathbf{X}$ and a collection of projective lines  $l_0,...,l_{n-1}$ в $\mathbf{X}$ such that $x_i,x_{i+1}\in l_i$.  

The variety of all length $n$ paths connecting $x$ and $y$
is denoted by $P_n(x,y)$.
\end{definition}

\begin{definition} Linear ind-variety 
$\mathbf X$  is called \textbf{\emph{$1$-connected}}, 
if for any two points $x,y\in \mathbf X$ there exists a path
connecting $x$ with $y$.
\end{definition}

\begin{definition} 
We will say that a vector bundle $\mathbf{E}$ on $\mathbf{X}$ is \textbf{trivial on lines} if for any projective line $l$ on $\mathbf{X}$ the restriction $\mathbf{E}|_{l}$ is trivial.
\end{definition}

\begin{definition} Let $\mathbf{E}$ be a rank $r$ bundle on a linear variety $\mathbf{X}$. The 
\textbf{splitting type} of the bundle  $\mathbf{E}$ on a projective line $ l \subset \mathbf{X}$  is a collection of numbers $r_i >0$ and $a_i \in \mathbb{Z}$, $i=1,...,s$ such that
$$\mathbf{E}|_l\cong r_1\mathcal{O}_{\mathbb{P}^1}(a_1) \oplus r_2\mathcal{O}_{\mathbb{P}^1}(a_2) \oplus \ldots r_s\mathcal{O}_{\mathbb{P}^1}(a_s)$$
и $a_{1}>a_{2}>...>a_s$, $\sum_{i=1}^{s}r_i=r$.

A bundle  $\mathbf{E}$ is called
{\textbf{uniform}}, if 
its restriction to all projective lines 
has the same splitting type.
\end{definition}
 
We will need to use several  results of articles
\cite{ermakova} and \cite{ermakova2} on complete 
intersections $\mathbf X \subset \mathbf{G}$,
that we recall now for convenience.


\begin{theorem}[\cite{ermakova}] \label{main1} Let $X$  
be a complete intersection of $G(n,2n)$ embedded by Pl{\"u}cker with a collection of hypersurfaces of degrees $d_1,...,d_l:$
$X=G(n,2n)\cap  \bigcap_{i=1}^{l} Y_{i}$.
If $2\sum_i (d_i+1)\le [\frac{n}{2}],$ then the variety $P_n(u,v)$ of length $n$ paths
connecting any two points $u,v$ in $X$
is non-empty and connected.
\end{theorem}

We will also need a corollary of this theorem for the 
case of complete intersection in $G(k,n)$.
We will assume $k\leq [\frac{n}{2}]$. 

\begin{corollary} \label{svjaznost' dlja G(k,n)} 
Let $X$ 
be a complete intersection of $G(k,n)$ embedded by Pl{\"u}cker 
with a collection of hypersurfaces of degrees $d_1,...,d_l:$ $X=G(k,n)\cap  \bigcap_{i=1}^{l} Y_{i}$.
If $2\sum_i (d_i+1)\le \frac{k}{2} \le [\frac{n}{4}]$ then the variety $P_k(u,v)$ of length $k$ paths
connecting any two points $u,v$ in $X$
is non-empty and connected.
\end{corollary}

\begin{proof}[\textbf{Proof.}]
We will assume that points $u$ and $v$ in 
$G(k,n)$ are generic \footnote{i.e., the intersection of the corresponding $k$-planes is $0$.};
in the case when the points are not generic the proof 
goes in the same way the proof of Theorem  \ref{main1}, see \cite{ermakova}.

Let  $U$ and $V$ be the $k$-dimensional spaces corresponding to the points $u$ and $v$ of $G(k,n)$. To construct a path $p_k(u,v)$ connecting $u$ and $v$ consider the Grassmannian $G(k,2k)=G(k,U \oplus V)$ of  $k$-dimensional spaces in  the $2k$-dimensional vector space $U \oplus V$. 
It is easy to prove that any path  $p_k(u,v)$ of length $k$ in the Grassmannian $G(k,n)$ is contained in the Grassmannian $G(k,U \oplus V)$.  So all length $k$ paths connecting $u$ and $v$
in $X$ are contained in $G(k,U \oplus V)$.

Consider now the intersection $G(k,U \oplus V)\cap \bigcap_{i=1}^lY_{i} \subset X$. By Theorem \ref{main1} the variety of paths $P_k(u,v)$ of the intersection $G(k,U \oplus V)\cap \bigcap_{i=1}^lY_{i}$ is non-empty and connected. It follows that the variety of paths $P_k(u,v)$ in $X$ is non-empty and connected as well.
\end{proof}

\begin{proposition}[\cite{ermakova2}]\label{segre1} Consider the Segre embedding of  $\mathbb P^l\times \mathbb P^{m}$ in $\mathbb P^{(l+1)(m+1)-1}$. Chose natural numbers  $(k,d)$ such that $2kd<\min(l,m)$. Then for any variety  $Y$ in $\mathbb P^{(l+1)(m+1)-1}$
whose irreducible components have codimension at most $k$ and degree at most $d$ the variety  $Y\cap \mathbb P^l\times \mathbb P^{m}$ is  $1$-connected.
\end{proposition}

\begin{lemma} [\cite{ermakova2}]\label{newPinf} Let $Y_1,...,Y_l$ be hypersurfaces of degrees $d_1,...,d_l$ in $\mathbb P^n$.
Let $Y=Y_1\cap...\cap Y_l$ and let $\mathbb P^k\subset Y$ be a $k$-dimensional projective subspace. There exists a projective subspace $\mathbb P^{k+1}\subset Y$ containing $\mathbb P^k$ if the following holds: $$n-k-1>\sum_{i=1}^l\binom{d_i+k}{d_i-1}.$$
\end{lemma}

\begin{theorem}[\cite{ermakova2}]\label{uniform}
Any finite rank vector bundle  $\mathbf{E}$ on $\mathbf X$ is uniform.
\end{theorem}

Finally we list the steps in our proof of Theorem
\ref{MainTheorem}.

\begin{itemize}

\item We  prove first that any finite rank vector bundle  $\mathbf E$ 
contains a flag of subbundles  $0=\mathbf{F_0}\subset \mathbf{F_1}\subset \mathbf{F_2}\subset ...\subset \mathbf{F_s}=\mathbf{E}$
such that each quotient $\mathbf{F_i}/\mathbf{F_{i-1}}$
is a bundle trivial on lines twisted by $\mathcal O(a_i)$ for   $1\leq i \leq s$.

\item  Next we prove that every finite rank bundle on 
$\mathbf X$ trivial on lines is trivial.

\item Finally, using Kodaira vanishing theorem we  prove that the bundle  $\mathbf{E}$ splits as a sum $\oplus_i \mathbf{F_i}/\mathbf{F_{i-1}}$.
  
\end{itemize}

\section{Constructing a flag of subbundles in $\mathbf{E}$}

In this section we will construct a flag of subbundles in a rank $r$ vector bundle $\mathbf{E}$ on a complete intersection $\mathbf{X} \subset \mathbf{G}$ of finite codimension. 

Before doing this formally we will describe the main idea.

Chose a point $x\in \mathbf{X}$ and consider the fibre $\mathbf{E}_x$ of  $\mathbf{E}$ over $x$. According to Grothendieck's theorem (\cite{OSS} Theorem 2.1.1), 
for any projective line $l$ passing through $x$ the restriction
$\mathbf{E}|_l$ has a canonical flag of subbundles
$0=\mathbf{F}_0\subset \mathbf{F}_1\subset ... \subset \mathbf{F}_s=\mathbf{E}|_l$. 

Hence, we get in  $\mathbf{E}_x$ a flag of subspaces that we denote by $\mathbf{F}(x,l)$. A priori the flag
 $\mathbf{F}(x,l)$ might depend on the choice of  $l$ passing through $x$ but we will show that this is not the case.

Let $B_m(x)$ be the \textit{base of the family of lines} on $X_m$ passing through $x$, considered as a reduced scheme. We will show that the map from $B_m(x)$ to the space of flags $\mathbf{E}_x$
that associates to each line  $l \subset B_m(x)$ the flag $\mathbf{F}(x,l)$ is a morphism for all $m$. After that we will apply the following theorem.
\begin{theorem} \label{TheoremBaseInPoint}
For any $d$ there exists a number $M:=M(d)$ such that for any $m>M$ 
any morphism from $B_m(x)$ to a projective variety of dimension 
less than $d$ is constant.
\end{theorem}

\subsection{Proof of Theorem \ref{TheoremBaseInPoint}}

We will start by analysing $B_m(x)$. Recall that for $i=1,...,l$ the number $d_i=\deg Y_{i,m}$, ${m\geq1}$ denotes the degree of the
hypersuface ${Y_{i,m}}$.
Denote by $\mathbb{P}_{x}^{N_m-2}$  the projectivised tangent 
space  to the Pl{\"u}cker space $\mathbb{P}^{N_m-1}$
at point $x$.

\begin{proposition} \label{baza} 
$B_m(x)$ is given by the intersection of $\mathbb{P}^{k_m-1} \times \mathbb{P}^{n_m-k_m-1}$ with
$\sum_id_{i}$ hypersurfaces in  $\mathbb{P}_{x}^{N_m-2}$ of degrees
no more than $\max_i(d_i)$. 
\end{proposition}

\begin{proof}[\textbf{Proof.}]
The base of the family of lines passing through 
$x$ on $G(k_m,n_m)$ is $\mathbb{P}^{k_m-1} \times  \mathbb{P}^{n_m-k_m-1}$. So it is enough to prove that for any $i$
the space of lines on $Y_{i,m}$  
passing through  $x$ is the intersection of $d_{i}$ hypersurfaces 
in $\mathbb{P}_{x}^{N_m-2}$.

Let us introduce homogeneous coordinates $(z_0:z_1:...)$ on $\mathbb{P}^{N_m-1}$. Let $x=(1:0:..:0)$,
then $(z_1:z_2:...)$ are homogeneous coordinates on
$\mathbb{P}_{x}^{N_m-2}$. Let us write the equation of 
the hypersuface $Y_{i,m}$ in the form
$$F=F_{d_i}(z_1:z_2:...)+z_0F_{d_i-1}(z_1:z_2:...)+...+z_0^{d_i-1}F_{1}(z_1:z_2:...)=0.$$
Then the variety of lines on $Y_{i,m}$ passing through $x$
is given in   $\mathbb{P}_{x}^{N_m-2}$
by a system of $d_i$ equations
$$F_{d_i}(z_1:z_2:...)=F_{d_i-1}(z_1:z_2:...)=...=0$$
of degrees $d_i,d_i-1,...,1$ correspondingly. 
This proves our claim.

\end{proof}

\begin{proof} [\textbf{Proof of Theorem \ref{TheoremBaseInPoint}.}]
Note that for any $d$ 
there is a number $M$
such that for any $m>M$
the following two conditions hold:

1) The variety $B_m(x)$ is  $1$-connected. 
This follows from Proposition \ref{segre1}
together with Proposition \ref{baza}.

2) Any projective line on $B_m(x)$
is contained in a projective subspaces 
$\mathbb{P}^d \subset B_m(x)$ of dimension $d$. 
This follows from Lemma \ref{newPinf}. 

From conditions 1) and 2) it follows that any morphism from $B_m(x)$ 
to any variety of dimension less than $d$ is a morphism to a point.
Indeed any line in $B_m(x)$ is mapped to a point since
it is contained in some  $\mathbb P^d$ (which in its turn has to be mapped to a point by \cite{H}, section II, $\S$ 7, exercise 7.3a).
Since any two points in $B_m(x)$ are connected by a chain 
of lines, the whole variety $B_m(x)$ is mapped to a point as well.
\end{proof}

\subsection{A standard lemma}

Further on we will need the following standard fact.
\begin{lemma}\label{lemmaBaseInPoint}  Let $X$, $Y$, $Z$ be projective varieties and suppose that $Y$ is smooth.
Suppose we have morphisms $p:X\to Y$ and $\pi:X\to Z$
such that the fibres of the morphism $p$ 
are contained in the fibres of the morphism $\pi$.
Then there is a morphism $f:Y\to Z$ such that $f \circ p = \pi$.
\[
\xymatrix{& X\ar[r]^{p}\ar[d]_{\pi} & Y\ar@{-->}[dl]^{f}\\
& Z }
\]
\end{lemma}

\begin{proof}[\textbf{Proof.}]
Consider the morphism $\phi: X\to Y\times Z$, $\phi(x)=(p(x),\pi(x))$ and denote by $p'$, $\pi'$
the projections of $Y\times Z$ on $Y$ and $Z$ correspondingly.
\[
\xymatrix{& X\ar[r]^{p}\ar[d]_{\pi}\ar[dr]_{\phi} & Y\\
& Z & Y \times Z \ar[l]^{\pi'} \ar[u]_{p'} }
\]
Note that the projection 
$p'_1: \phi(X)\to Y$ is an isomorphism (\cite{Sh}, section II.4, Theorem 2) since by our assumptions the projection $p'$
is a bijective morphism and  $Y$ is smooth. The desired 
morphism  $f:Y\to Z$ is given by the composition
$f=\pi'\circ p'^{-1}$.
\end{proof}

\subsection{Constructing a flag of subbundles in $\mathbf{E}$ on $X_m$}
Finally, we  start to construct the flag of subbundles.
Let  $B_m$ be the base of the family of lines on $X_m$.
Let us consider the following set
$$\Gamma_m=\{(x,l)\in X_m \times B_m|x \in l\}$$
as a reduced scheme. Denote by $p_m: \Gamma_m \rightarrow X_m$ the morphism such that $p_m(x,l)= x$.

Recall that $\mathbf{E}|_{X_m}=E_m$ and the rank of $E_m$ is $r$. 
The fibre at  point $x \in X_m$ is denoted by  $E_m(x)$.

According to Theorem \ref{uniform}
the bundle $E_m$ is uniform. So for any line $l \in B_m$ we have 
the following splitting 
$$E_m|_l=r_1\mathcal{O}_{\mathbb{P}^1}(a_1) \oplus r_2\mathcal{O}_{\mathbb{P}^1}(a_2) \oplus ...\oplus r_s\mathcal{O}_{\mathbb{P}^1}(a_s), \: a_1>a_2>...>a_s, \: \sum_{i=1}^{s}r_i=r,$$ where
$r_i$ and $a_i$ do not depend on the choice of the line $l$.

\begin{theorem} \label{ChainOfSubbandels} Let $\mathbf{X}=\underrightarrow{\lim}X_m$ be a complete intersection
of codimension $l$ in the linear ind-Grassmannian $\mathbf{G}$.
For any $m\geq 1$ the bundle $E_m$ on $X_m$ has 
a flag of subbundles 
$$0=F_0 \subset F_1 \subset F_2 \subset \ldots \subset F_s = E_m$$
such that for any line $l \in B_m$
\begin{equation}\label{F_i|l} F_i|_l=r_1\mathcal{O}_{\mathbb{P}^1}(a_1) \oplus \ldots \oplus r_i\mathcal{O}_{\mathbb{P}^1}(a_i), \: 1\leq i \leq s.
\end{equation}
\end{theorem}

\begin{proof}[\textbf{Proof.}]
It is sufficient to prove this statement for all 
$m$ from Theorem \ref{TheoremBaseInPoint} that are larger than
$M(2r)$. Let $E_{(1)}:=E_m$.
Let
$$\mathfrak{Gr} (r_1,E_{(1)})=\bigcup_{x \in X_m} G(r_1,E_{(1)}(x))$$
be the grassmannisation of the bundle $E_{(1)}$
with its natural projection  $\varphi_1$ $:\mathfrak{Gr} (r_1,E_{(1)}) \rightarrow X_m$, where $\varphi_1:(x,r_1\mathcal{O}_{\mathbb{P}^1}(a_1)|_x)\mapsto x$.
Consider the morphism $\pi_1:\Gamma_m\rightarrow \mathfrak{Gr}(r_1,E_{(1)})$ where $\pi_1:(x,l)\mapsto (x,r_1\mathcal{O}_{\mathbb{P}^1}(a_1)|_x)$.
Let us show that the fibres of the morphism
$p_m:(x,l)\mapsto x$ are contained in the fibres of the morphism $\pi_1$. Indeed, for any point $x \in X_m$ the fibre $p_m^{-1}(x)$ of the morphism $p_m$
is isomorphic to the base of the family of lines passing though $x$
on  $X_m$. The morphism $\pi_1$ sends $p_m^{-1}(x)$ to $\mathfrak{Gr}(r_1,E_{(1)}(x))$ and this map is a map to a point according to Theorem \ref{TheoremBaseInPoint} since $\dim \mathfrak{Gr}(r_1,E_{(1)}(x))<2r$.

Consider the following diagram:
\[
\xymatrix{\Gamma_m \ar[r]^{p_m}\ar[d]_{\pi_1} & X_m \ar@{-->}[dl]_{f_1}\\
\mathfrak{Gr}(r_1,E_{(1)})\ar@/_1pc/[ur]_{\varphi_1}}
\]
The existence of the section  $f_1$ of the projection $\varphi_1$ follows from Lemma \ref{lemmaBaseInPoint} in which we set $X=\Gamma_m$, $Y=X_m$, $Z=\mathfrak{Gr}(r_1,E_{(1)})$, $p=p_m$, $\pi=\pi_1$.

Denote by $\mathfrak{S}_{r_1}$   the tautological $r_1$-dimensional  subbundle in $\varphi_1 ^*E_{(1)}$.
Applying the functor $f^*_1$ to the monomoprphism of bundles  
 $$\tau_1: \mathfrak{S}_{r_1} \longrightarrow \varphi_1 ^*E_{(1)},$$
we get the following monomorphism of bundles:
$$f_1^* \tau _1:f_1^*\mathfrak{S}_{r_1} \rightarrow E_{(1)}.$$
Denote $E_{(2)}=E_{(1)}/_{f_1^*\mathfrak{S}_{r_1}}$,  
and consider the grassmannisation 
$$\mathfrak{Gr} (r_2,E_{(2)})=\bigcup_{x \in X_m} G(r_2,E_{(2)}(x))$$
of the bundle $E_{(2)}$.

We get the following diagram
\[
\xymatrix{\Gamma_m \ar[r]^{p_m}\ar[d]_{\pi_2} & X_m \ar@{-->}[dl]_{f_2}\\
\mathfrak{Gr}(r_2,E_{(2)})\ar@/_1pc/[ur]_{\varphi_2}},
\]
where the existence of the morphism $f_2$ follows from Lemma \ref{lemmaBaseInPoint} in the same way as the existence of 
the morphism  $f_1$. By the same considerations as before 
we get the embedding
$$f_2^* \tau _2:f_2^*\mathfrak{S}_{r_2} \rightarrowtail E_{(2)}.$$
Denote
$$E_{(3)}=E_{(2)}/_{f_2^*\mathfrak{S}_{r_2}}.$$

Repeating our previous considerations we get a family of 
epimorphisms of bundles:
$$E_{(1)} \twoheadrightarrow E_{(2)} \twoheadrightarrow \ldots \twoheadrightarrow E_{(s+1)}.$$ 
We associate to it a family of subbundles 
in  $E_{(1)}$:
$$0=F_0 \subset F_1 \subset F_2 \subset \ldots \subset F_s = E_{(1)}$$
such that for any line $l \in B_m$
$$F_i|_l \cong r_1\mathcal{O}_{\mathbb{P}^1}(a_1) \oplus \ldots \oplus r_i\mathcal{O}_{\mathbb{P}^1}(a_i), \: 1\leq i \leq s.$$
\end{proof}

\begin{corollary} \label{F_i/F_i-1=r_iO(a_i)} Any bundle $F_{i}/F_{i-1}$ is  a bundle trivial on lines
twisted by the line bundle $\mathcal{O}_{X_m}(a_i)$.
\end{corollary}
Indeed, from the construction of the flag of subbundles 
$E_m$ it follows that the restriction of  $F_{i}/F_{i-1}$ to any line $l \in B_m$ is equal to
$r_i\mathcal{O}_{\mathbb{P}^1}(a_i)$ for $1\leq i\leq s$.
In other words the bundle $F_{i}/F_{i-1}\otimes \mathcal{O}_{X_m}(-a_i)$ is  trivial on lines.

\begin{zam}Subbundles $F_i$ of the bundle $E_m$ in Theorem \ref{ChainOfSubbandels}  are defined in a unique way. This follows from the fact
that any vector bundle $E$ on a projective line $\mathbb{P}^1$ has a canonically defined filtration $0 = F_0 \subset F_1 \subset F_2 \subset ... \subset F_s=E_m$ such that $F_{i}/F_{i-1}$ is isomorphic to $r_iO(a_i)$ where $a_{1}>a_{2}>...>a_s$.
\end{zam}

Using the fact that the flag of subbundles 
$0 = F_0 \subset F_1 \subset F_2 \subset ... \subset F_s$
constructed in Theorem \ref{ChainOfSubbandels} does not depend on $m$,
and using the linearity of the ind-variety $\mathbf{X}=\underrightarrow{\lim}X_m$ 
we get the following corollary from Theorem \ref{ChainOfSubbandels}.

\begin{theorem} \label{podkrutka} Let $\mathbf{X}$ 
be a complete intersection in the linear ind-Grassmannian  $\mathbf{G}$ and let $\mathbf{E}$ be a uniform bundle on $\mathbf{X}$. 
Then there exists a flag of subbundles 
$$0 =\mathbf{F_0} \subset \mathbf{F_1} \subset \ldots \subset \mathbf{F_s} = \mathbf{E}$$
such that any quotient bundle 
$\mathbf{F_{i}}/\mathbf{F_{i-1}}$ is a  bundle trivial on lines twisted by a line bundle.
\end{theorem}

\section{A criterion of triviality of  bundles trivial on lines}

The goal of this section is to prove Theorem
\ref{linear_trivialiti} which gives a criterion for 
a  bundle trivial on lines to be trivial.

Let $X$ be a normal projective variety and $E$ a be vector bundle on $X$.
Let $Y$ be the Fano scheme of lines  on $X$.
Let $Z \subset X\times Y$ be the universal line.
Denote the projections to $Y$ and to $X$ by $\pi$ and $p$ respectively.
Note that $\pi:Z \to Y$ is a $\mathbb{P}^1$-bundle.
Consider the scheme

\begin{equation*}
Z_1 = Z \times_Y Z.
\end{equation*}

It parameterizes lines on $X$ with pair of points on them.
Let $p_1,p_2:Z_1 \to X$ be the compositions of the projections $p_{r_1},p_{r_2}:Z_1 \to Z$ with the map $p:Z \to X$.
Further, we define inductively the variety

\begin{equation*}
Z_{n+1} = Z_n \times_X Z_1,
\end{equation*}
with projections $p_1:Z_1 \to X$ и $p_{n+2}:Z_{n+1} \to X$, using the following diagram:


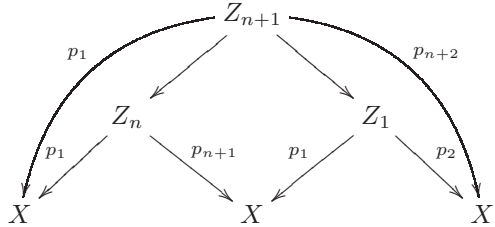
\begin{equation*}
\xymatrix{
&& Z_{n+1} \ar[dl] \ar[dr] \ar@/_2pc/[ddll]_{p_1} \ar@/^2pc/[ddrr]^{p_{n+2}} \\
& Z_n \ar[dl]_{p_1} \ar[dr]^{p_{n+1}} && Z_1 \ar[dl]_{p_1} \ar[dr]^{p_2} \\
X && X && X
}
\end{equation*}

Finally, for a point $x \in X$ we define

\begin{equation*}
Z_n(x) := p_1^{-1}(x).
\end{equation*}

The projection $Z_n(x) \to X$ induced by the projection $p_{n+1}:Z_n \to X$ will be denoted by $f_{x,n}$.

\begin{lemma} \label{lemma_trivialnost}
Let $E$ be a vector bundle on $X$ trivial on lines. Then on $Z_1$ we have an isomorphism $p_1^*E \cong p_2^*E$. 
\end{lemma}
\begin{proof}[\textbf{Proof.}]
First consider the vector bundle $p^*E$ on $Z$. Since $E$ is trivial on lines, it is trivial on all fibres of $\pi:Z \to Y$.
Since the latter is a $\mathbb{P}^1$-bundle, it follows that $p^*E \cong \pi^*F$ for some vector bundle $F$ on $Y$. Now consider the diagram
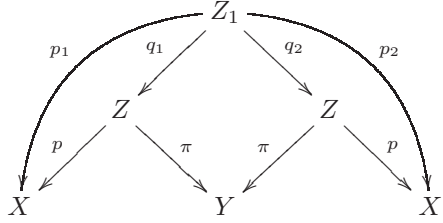
\begin{equation*}
\xymatrix{
&& Z_1 \ar[dl]_{q_1} \ar[dr]^{q_2} \ar@/_2pc/[ddll]_{p_1} \ar@/^2pc/[ddrr]^{p_{2}} \\
& Z \ar[dl]_p \ar[dr]^\pi && Z \ar[dl]_\pi \ar[dr]^p \\
X && Y && X
}
\end{equation*}
We have
\begin{equation*}
p_1^*E = q_1^*p^*E \cong q_1^*\pi^*F \cong q_2^*\pi^*F \cong q_2^*p^*E \cong p_2^*E,
\end{equation*}
which proves the Lemma.
\end{proof}

The next step is the following

\begin{lemma}
If $E$ is trivial on lines then for each $n > 0$ the bundle $f_{x,n}^*E$ on $Z_n(x)$ is trivial.
\end{lemma}
\begin{proof}[\textbf{Proof.}]
We argue by induction on $n$. For $n = 1$ we have
\begin{equation*}
f_{x,1}^*E = (p_2^*E)|_{Z_1(x)} \cong (p_1^*E)|_{Z_1(x)} \cong E_x\otimes \mathcal{O}_{Z_1(x)}
\end{equation*}
since the composition $Z_1(x) \subset Z_1 \xrightarrow{p_1} X$ factors through the point $x$.
This justifies the base of the induction.

Now assume that the claim is true for some $n$. Then for $n+1$ consider the diagram
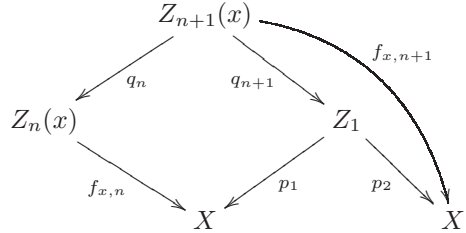
\begin{equation*}
\xymatrix{
&& Z_{n+1}(x) \ar[dl]^{q_n} \ar[dr]_{q_{n+1}} \ar@/^2pc/[ddrr]^{f_{x,n+1}} \\
& Z_n(x) \ar[dr]_{f_{x,n}} && Z_1 \ar[dl]^{p_1} \ar[dr]_{p_2} \\
&& X && X
}
\end{equation*}
We have
\begin{equation*}
f_{x,n+1}^*E = q_{n+1}^*p_2^*E \cong q_{n+1}^*p_1^*E \cong q_n^*f_{x,n}^*E \cong q_n^*\mathcal{O}_{Z_n(x)}^{\oplus r} \cong \mathcal{O}_{Z_{n+1}(x)}^{\oplus r},
\end{equation*}
which proves the Lemma.
\end{proof}

Now we can finish by the following argument

\begin{theorem} \label{linear_trivialiti}
Assume that $X$ is normal and for some $n > 0$ and some point $x \in X$ the map $f_{x,n}:Z_n(x) \to X$ is dominant and has connected fibers.
Then any vector bundle on $X$ trivial on all lines is trivial.
\end{theorem}

\begin{proof}[\textbf{Proof.}]
Assume that $f_{x,n}$ is dominant and has connected fibers.
Then $(f_{x,n})_*\mathcal{O}_{Z_n(x)} \cong \mathcal{O}_X$ since $X$ is normal.
Hence, by projection formula we have
\begin{equation*}
(f_{x,n})_*f_{x,n}^*E \cong E\otimes (f_{x,n})_*\mathcal{O}_{Z_n(x)} \cong E\otimes\mathcal{O}_X \cong E.
\end{equation*}

Finally, by Lemma \ref{lemma_trivialnost} we have $f_{n,x}^*E \cong \mathcal{O}_{Z_n(x)}^{\oplus r}$, hence
\begin{equation*}
(f_{x,n})_*f_{x,n}^*E \cong
(f_{x,n})_*\mathcal{O}_{Z_n(x)}^{\oplus r} \cong \mathcal{O}_X^{\oplus r}.
\end{equation*}
Comparing these two equalities we see that $E$ is trivial.

\end{proof}


\section{Splitting of the bundle $\mathbf{E}$}

To finish the proof of splitting of the  bundle
$\mathbf{E}$ on the variety $\mathbf{X}$ we apply 
Kodaira vanishing theorem \cite{GH}.

\begin{theorem} Let  $X$ be a smooth projective variety and let
$L$ be an ample line bundle on it. Then for any 
$q>0$  we have $H^q(X,K_X\otimes L)=0$.
\end{theorem}

Recall the following standard fact.
\begin{theorem} Let $X$ be a complex projective variety
and let $E$ be a vector bundle on $X$ with a subbundle $F$. Suppose that $H^1((E/F)^*\otimes F)=0$
then  $E\cong F\oplus E/F$.
\end{theorem}

Recall as well the formula for the canonical class of the Grassmannian
$G(k,n)$: $$K_{G(k,n)}=\mathcal{O}_{G(k,n)}(-n).$$

Recall the adjunction formula
\begin{theorem} Let $X$ be a smooth variety that is a 
complete intersection of  $G(k,n)$ with a finite 
collection of smooth hypersurfaces 
$Y_1, \ldots , Y_l$ of degrees $d_1,...,d_l$ correspondingly,
and let $d=\sum_{i=1}^k d_l$.
Then we have $K_{X}=K_{G(k,n)}\otimes \mathcal{O}(d_1+...+d_l)=\mathcal{O}(d-n)$.
\end{theorem}

Recall finally that in the case when the number $l$ of hypersurfaces is less than $\dim G(k,n)-2$ (i.e., $\dim X>2$)
by Lefschetz hyperplane theorem (\cite{La}, Theorem 3.1.17)
we have $Pic(X)=Pic(G(k,n))=\mathbb{Z}$.

\begin{corollary} \label{dimL=r>d-n} Consider on $X$ the bundle $L=\mathcal{O}(r)$. If  $r>d-n$ then  $H^1(L)=0$.

$H^1(\mathcal{O}_X(d-n+r))=0$ for all $ r>0$. In particular for $n\gg 1, \ d=2$ we have $H^1(\mathcal{O}_X(a))$ for all $ a>0$.
\end{corollary}

\begin{theorem} \label{E-rassheplenie v prjamuju summu} Let  $\mathbf{X}$ be an ind-variety and $\mathbf{E}$ be a vector bundle on it. Let $0=\mathbf{F_0}\subset \mathbf{F_1}\subset ...\subset \mathbf{F_s}=\mathbf{E}$ be a flag of subbundles such that each bundle $\mathbf{F_{i}}/\mathbf{F_{i-1}}$ is a twist of a  
 bundle trivial on lines by a line bundle. Then $\mathbf{E}=\oplus_i\mathbf{F_{i}}/\mathbf{F_{i-1}}$.
\end{theorem}

\begin{proof}[\textbf{Proof.}]
Let us show that there is $M \in \mathbb{Z_+}$ such that for any $m>M$ the restriction of the bundle  $\mathbf{E}$ on $X_m$
splits as a sum  $\oplus_i\mathbf{F_{i}}/\mathbf{F_{i-1}}|_{X_m}$ of corresponding subbundles on $X_m$. Namely,
chose  $M$ such that $d-n_M<0$. We will establish this splitting by induction in $i$. Suppose it is proven that
$F_{i-1}=\sum_{1\le j\le {i-1}}r_j\mathcal{O}(a_j)$, $a_{j-1}>a_j$. Then $F_{i}$ is an extension of the bundle 
 $F_{i-1}$ by $r_{i}\mathcal{O}_{a_{i}}$. To prove that $F_{i}$
splits it is enough to know that
$$H^1(r_{i}\mathcal{O}(-a_{i})\otimes  \sum_{1\le j\le {i-1}}r_j\mathcal{O}(a_j) )=0.$$
The latter  holds by Corollary
\ref{dimL=r>d-n} since $a_{j}-a_{i}>0>d-n_M$
for any $j\le i$.

\end{proof}

\section{Proof of Theorem \ref{MainTheorem}}

In this section we prove Theorem \ref{MainTheorem}.

Since the bundle $\mathbf{E}$ is uniform 
we can apply Theorem \ref{podkrutka}. 
It follows that each quotient bundle
$\mathbf{F_{i}}/\mathbf{F_{i-1}}$ for $1\leq i\leq s$ 
is a twist of a bundle  trivial on lines by a line bundle.

Set in  Theorem \ref{linear_trivialiti} $n=k_m$ and apply it  to $X=X_m$ 
(a complete intersection in  $G(k_m,n_m)$). 
Then the fibre of the morphism 
$f_{x,k_m}$ over $y \in X_m$ 
is the space of paths of length $k_m$  on $X_m$ that start at ${x}$ and finish at $y$.
In Corollary \ref{svjaznost' dlja G(k,n)} set $X=X_m$ (recall that  $X_m=G(k_m,n_m)\cap \bigcap_{i=1}^l Y_{i,m}$). 
Then for a sufficiently large  $m$ such that $k_m$ and $n_m$
satisfy the inequality   $2\sum_i (d_i+1)\le \frac{k_m}{2} \le [\frac{n_m}{4}]$ where $d_i=\deg Y_{i,m}$, we deduce that the space of paths $P_{k_m}(x,y)$ connecting $x$ with $y$ on  $X_m$ is non-empty and connected. 
Hence the fibres of the morphism  $f_{x,k_m}$ are non-empty and connected, in particular $f_{x,k_m}$ is dominant.
 
So  the conditions of Theorem  \ref{linear_trivialiti} hold for $X_m$. 
It follows that any  vector bundle on $\mathbf{X}$  trivial on lines 
is trivial.
So each $\mathbf{F_{i}}/\mathbf{F_{i-1}}$ is a twist of 
a trivial bundle by a line bundle.

Finally, using Theorem  \ref{E-rassheplenie v prjamuju summu} we deduce that the bundle $\mathbf{E}$
is a direct sum $\mathbf{E}=\oplus_i\mathbf{F_{i}}/\mathbf{F_{i-1}}$.
This proves Theorem \ref{MainTheorem}.

\hfill $\square$

\bigskip

\null

\end{document}